\documentclass[11pt]{amsart}

\usepackage[a4paper,margin=1in]{geometry}
\usepackage{amsmath,amssymb,amsthm,mathtools,booktabs}
\usepackage{enumitem}
\usepackage[colorlinks=true,linkcolor=blue,citecolor=blue,urlcolor=blue]{hyperref}
\newtheorem{theorem}{Theorem}[section]
\newtheorem{lemma}[theorem]{Lemma}
\newtheorem{proposition}[theorem]{Proposition}
\newtheorem{corollary}[theorem]{Corollary}
\theoremstyle{remark}
\newtheorem{remark}[theorem]{Remark}

\newcommand{\Tr}{\operatorname{Tr}}
\newcommand{\Ree}{\operatorname{Re}}

\newcommand{\KF}[2]{\|#1\|_{(#2)}}
\newcommand{\UIN}[1]{\lvert\!\lvert\!\lvert #1\rvert\!\rvert\!\rvert}

\title[Crossed-Heinz inequalities]{Crossed-Heinz Inequalities:
	The Two-Dimensional Threshold and Higher-Dimensional Stability}
\author{Mohammad Sababheh}
\address{Program of Mathematics, Abdullah Al Salem University, Khaldiya, Kuwait City, Kuwait}
\email{sababheh@yahoo.com}
\author{Jean-Christophe Bourin}
\address{Université Marie et Louis Pasteur, CNRS, LmB (UMR 6623), F-25000 Besançon, France
	Corresponding authors}
\email{jcbourin@univ-fcomte.fr}
\date{}

\begin{document}
	
	\maketitle
	
	\begin{abstract}
		
		For positive matrices $A,B$ and $0<t<1$, Bourin asked whether the
		crossed-Heinz expression
		\[
		C_t(A,B)=A^tB^{1-t}+B^tA^{1-t}
		\]
		satisfies
		\[
		\UIN{C_t(A,B)}\le \UIN{A+B}
		\]
		for every unitarily invariant norm.  The most direct route to this
		inequality is to compare $C_t(A,B)$ with the ordinary Heinz expression
		\[
		H_t(A,B)=A^tB^{1-t}+A^{1-t}B^t,
		\]
		because the classical Heinz inequality gives
		$\UIN{H_t(A,B)}\le\UIN{A+B}$.  This route, however, encounters a
		genuine dimensional obstruction: the stronger comparison
		$\UIN{C_t(A,B)}\le\UIN{H_t(A,B)}$ is known to fail for positive
		definite $3\times3$ matrices, and hence, by block-diagonal embedding,
		in every dimension $n\ge3$.
		
		We first determine the exact threshold for this obstruction.  For every
		pair of positive $2\times2$ matrices and every $0<t<1$, we prove the
		weak-majorization relation
		\[
		s(C_t(A,B))\prec_w s(H_t(A,B)).
		\]
		Thus the stronger crossed-Heinz/Heinz comparison holds for every
		unitarily invariant norm in dimensions $n\le2$, while it fails in
		general for every $n\ge3$.  In particular, Bourin's inequality holds
		for all positive $2\times2$ matrices throughout the full parameter
		range.
		
		The failure of the Heinz comparison from dimension three onward does
		not, however, settle Bourin's original problem; it only blocks this
		natural proof strategy.  We therefore turn directly to the Bourin
		deficit in arbitrary dimension.  Near every positive definite matrix
		with simple spectrum, we obtain an exact second-variation formula for
		each Ky Fan deficit and use its transverse positivity to prove local
		validity of Bourin's inequality.  In a second higher-dimensional
		regime, when $A$ and $B$ have a common two-point spectrum, we prove the
		sharp Loewner refinement
		\[
		|C_t(A,B)|+\gamma_t(A-B)^2\preceq A+B
		\]
		and determine the largest admissible coefficient $\gamma_t$.
		Accordingly, the paper identifies both the precise dimensional limit
		of the ordinary-Heinz route and two mechanisms by which direct
		Bourin-type domination survives beyond that obstruction.
		
	\end{abstract}
	
	\section{Introduction}
	
	Let $M_n=M_n(\mathbb C)$ and let $A,B\in M_n$ be positive definite.
	For $0<t<1$, set
	\begin{equation}\label{eq:C-def}
		C_t(A,B)=A^tB^{1-t}+B^tA^{1-t}.
	\end{equation}
	Bourin's question \cite{Bourin} asks whether
	\begin{equation}\label{eq:Bourin}
		\UIN{C_t(A,B)}\le \UIN{A+B}
	\end{equation}
	holds for every unitarily invariant norm.  Bourin originally stated
	the problem in the equivalent homogeneous form
	\[
	\UIN{A^pB^q+B^pA^q}
	\le
	\UIN{A^{p+q}+B^{p+q}},
	\qquad p,q\ge0.
	\]
	When $p+q>0$, one obtains \eqref{eq:Bourin} by taking
	$t=p/(p+q)$ and replacing the matrices by suitable powers.  Also,
	\[
	C_{1-t}(A,B)=C_t(A,B)^*,
	\]
	so the problem is symmetric about $t=1/2$.
	
	The question is particularly natural because of its proximity to the
	classical Heinz inequality.  Introduce
	\begin{equation}\label{eq:H-def}
		H_t(A,B)=A^tB^{1-t}+A^{1-t}B^t.
	\end{equation}
	For every unitarily invariant norm,
	\begin{equation}\label{eq:Heinz}
		\UIN{H_t(A,B)}\le \UIN{A+B},
		\qquad 0\le t\le1;
	\end{equation}
	see, for example, \cite{Bhatia,BhatiaDavis}.  Thus the seemingly most
	natural route to Bourin's inequality is the stronger comparison
	\begin{equation}\label{eq:strong-question}
		\UIN{C_t(A,B)}\le \UIN{H_t(A,B)}.
	\end{equation}
	If \eqref{eq:strong-question} were true, then \eqref{eq:Bourin} would
	follow immediately from \eqref{eq:Heinz}.
	
	This stronger route was investigated in connection with work of
	Hayajneh and Kittaneh \cite{HKLieb} and Bhatia \cite{BhatiaTrace}.
	For the Hilbert--Schmidt norm, Bhatia proved the comparison on the
	central interval $[1/4,3/4]$.  Bottazzi, Elencwajg, Larotonda, and
	Varela \cite{BELV} developed the connection further, but also found a
	decisive obstruction: their Example~3.1 gives positive definite
	$3\times3$ matrices for which
	\[
	\|C_t(A,B)\|_\infty>\|H_t(A,B)\|_\infty
	\]
	for a parameter $t$ outside the central interval.  Thus the ordinary
	Heinz expression cannot dominate the crossed one in general.  Once a
	$3\times3$ counterexample is available, direct sums with a sufficiently
	small positive scalar block propagate the failure to every dimension
	$n\ge3$.
	
	It is important to separate this obstruction from Bourin's original
	question.  The failure of \eqref{eq:strong-question} says that the
	ordinary Heinz mean cannot serve as a universal intermediate bound; it
	does not provide a counterexample to \eqref{eq:Bourin}.  Indeed,
	substantial positive information is known for the original problem.
	The trace-norm inequality holds throughout the full parameter range,
	and a series of papers by Hayajneh, Hayajneh, Kittaneh, and
	collaborators obtained endpoint results, quantitative estimates, and
	further special cases; see \cite{HHKII,HHKIII,HHKRemarks,HHKResults}.
	Darweesh, Hayajneh, Hayajneh, and Kittaneh \cite{DHHK} obtained
	additional estimates under spectral comparability assumptions.
	Kittaneh and Ricard \cite{KR} then proved the constant-one inequality
	\begin{equation}\label{eq:central-known}
		\UIN{C_t(A,B)}\le\UIN{A+B},
		\qquad \frac14\le t\le\frac34,
	\end{equation}
	for arbitrary positive matrices and arbitrary unitarily invariant
	norms.  Related extensions to noncommutative symmetric spaces have
	also appeared; see \cite{LHZ,Zhang}.
	
	This history suggests two different questions.  First, if the Heinz
	route fails from dimension three onward, is dimension two exactly the
	last dimension in which that stronger mechanism survives?  Second,
	once the Heinz route is unavailable, can one still obtain direct
	Bourin-type domination from other structural features of the pair
	$(A,B)$?  These are the two themes of the present paper.
	
	Our first result answers the dimensional question sharply.  For every
	pair of positive $2\times2$ matrices and every $0<t<1$, we prove
	\begin{equation}\label{eq:intro-major}
		s(C_t(A,B))\prec_w s(H_t(A,B)).
	\end{equation}
	Consequently,
	\[
	\UIN{C_t(A,B)}\le\UIN{H_t(A,B)}\le\UIN{A+B}
	\]
	for every unitarily invariant norm.  Combining this with the
	$3\times3$ counterexample of \cite{BELV} and a block-diagonal embedding
	shows that the universal crossed-Heinz/Heinz comparison has the exact
	dimensional threshold
	\[
	n\le2.
	\]
	To the best of our knowledge, the weak-majorization statement
	\eqref{eq:intro-major}, and hence this sharp dimensional formulation,
	has not previously been recorded.  The proof is genuinely
	two-dimensional: after a simultaneous unitary reduction, an exact
	formula for the two singular values shows that replacing one Heinz
	summand by its transpose leaves one Euclidean component unchanged and
	can only decrease the other.
	
	Having located the obstruction to the Heinz route, we then return to
	Bourin's original inequality itself and deliberately avoid comparison
	with $H_t$.  Two higher-dimensional mechanisms emerge.
	
	First, we study the problem near the diagonal $A=B$.  For a positive
	definite matrix $A$ with simple spectrum, we compute the exact
	second-order expansion of every Ky Fan deficit
	\[
	\KF{A+B}{k}-\KF{C_t(A,B)}{k}.
	\]
	The resulting Hessian is positive but has nontrivial null directions.
	We identify those directions with an exact zero-deficit manifold and
	prove uniform positivity in the transverse directions.  This yields a
	local theorem for all Ky Fan norms, and therefore for every unitarily
	invariant norm.
	
	Second, we consider a global spectral class in arbitrary dimension.
	When
	\[
	A=bI+(a-b)P,
	\qquad
	B=bI+(a-b)Q,
	\]
	with $P,Q$ orthogonal projections, the crossed-Heinz defect collapses
	to the algebra generated by $P+Q$.  This reduction yields the stronger
	Loewner estimate
	\[
	|C_t(A,B)|+\gamma_t(a,b)(A-B)^2\preceq A+B,
	\]
	where $\gamma_t(a,b)>0$ is explicit and optimal over all projection
	pairs.  Thus, even though the universal Heinz comparison breaks down
	from dimension three onward, direct Bourin domination persists in a
	nontrivial higher-dimensional class and admits a quantitative defect.
	
	The contribution of the paper can therefore be summarized as follows.
	\begin{enumerate}[label=\textup{(\roman*)},leftmargin=2.2em]
		\item
		We prove the weak majorization \eqref{eq:intro-major} for all positive
		$2\times2$ matrices.  Together with the known $3\times3$ counterexample,
		this identifies $n=2$ as the exact upper dimension for universal
		validity of the stronger crossed-Heinz/Heinz comparison.
		
		\item
		In arbitrary dimension, we derive an exact second-variation formula
		for every Ky Fan deficit near a simple-spectrum positive matrix and use
		the geometry of its null space to establish local validity of Bourin's
		inequality.
		
		\item
		For matrices with a common two-point spectrum, we prove a global sharp
		Loewner refinement of Bourin's inequality and determine its optimal
		coefficient.
	\end{enumerate}
	
	The three results fit the same progression.  The first determines
	precisely where the most natural Heinz-based proof strategy stops.  The
	second and third then show that the original Bourin inequality can
	survive beyond that obstruction for reasons that are not inherited
	from the ordinary Heinz mean: local Ky Fan curvature in one case and
	projection algebra in the other.
	
	\section{Notation and preliminary facts}
	
	For $X\in M_n$, let
	\[
	|X|=(X^*X)^{1/2},
	\]
	and write
	\[
	s_1(X)\ge\cdots\ge s_n(X)\ge0
	\]
	for the singular values of $X$.  The $k$-th Ky Fan norm is
	\[
	\KF{X}{k}=\sum_{j=1}^k s_j(X),
	\qquad 1\le k\le n.
	\]
	For nonnegative vectors $x,y\in\mathbb R^n$, the notation
	$x\prec_w y$ denotes weak majorization:
	\[
	\sum_{j=1}^k x_j^\downarrow
	\le
	\sum_{j=1}^k y_j^\downarrow,
	\qquad 1\le k\le n.
	\]
	The Ky Fan dominance theorem \cite{Bhatia} gives
	\begin{equation}\label{eq:KF-dominance}
		s(X)\prec_w s(Y)
		\quad\Longleftrightarrow\quad
		\UIN{X}\le\UIN{Y}
	\end{equation}
	for every unitarily invariant norm.
	
	We shall use repeatedly the elementary identity
	\begin{equation}\label{eq:defect}
		A+B-C_t(A,B)
		=
		(A^t-B^t)(A^{1-t}-B^{1-t}).
	\end{equation}
	The right-hand side need not be Hermitian, so \eqref{eq:defect} does
	not itself imply an order inequality.
	
	We also record the trace-norm case.
	
	\begin{proposition}\label{prop:trace}
		For all positive definite matrices $A,B$ and every $0<t<1$,
		\begin{equation}\label{eq:trace-general}
			\|C_t(A,B)\|_1\le \Tr A+\Tr B.
		\end{equation}
	\end{proposition}
	
	\begin{proof}
		Schatten H\"older gives
		\[
		\|A^tB^{1-t}\|_1
		\le
		(\Tr A)^t(\Tr B)^{1-t},
		\]
		and similarly
		\[
		\|B^tA^{1-t}\|_1
		\le
		(\Tr B)^t(\Tr A)^{1-t}.
		\]
		Hence
		\[
		\|C_t(A,B)\|_1
		\le
		(\Tr A)^t(\Tr B)^{1-t}
		+(\Tr B)^t(\Tr A)^{1-t}.
		\]
		Scalar Young's inequality gives the result.
	\end{proof}
	
	\section{The two-dimensional threshold for the Heinz comparison}
	
	We now turn to the main dimensional result.  Although comparison with
	the ordinary Heinz expression fails in general dimension, we shall show
	that it holds for every unitarily invariant norm in dimension two.
	The proof rests on two elementary but specifically two-dimensional
	features: a simultaneous real reduction and an exact representation of
	the two singular values of a real $2\times2$ matrix.
	
	\subsection{A real reduction}
	
	The first ingredient is a phase reduction peculiar to dimension two:
	after diagonalizing one Hermitian matrix, the single off-diagonal phase
	of the other can be removed without disturbing the diagonal form.
	
	\begin{lemma}\label{lem:real-reduction}
		Let $A,B\in M_2$ be Hermitian.  There exists a unitary matrix $U$ such
		that $U^*AU$ is diagonal and $U^*BU$ is real symmetric with
		nonnegative off-diagonal entry.  If $A,B\ge0$, then the transformed
		matrices remain positive.
	\end{lemma}
	
	\begin{proof}
		By the spectral theorem, there exists a unitary matrix $V$ such that
		\[
		V^*AV=
		\begin{pmatrix}
			a&0\\
			0&b
		\end{pmatrix},
		\qquad a,b\in\mathbb R.
		\]
		Since $B$ is Hermitian, in the same basis we may write
		\[
		V^*BV=
		\begin{pmatrix}
			c&re^{i\theta}\\
			re^{-i\theta}&d
		\end{pmatrix},
		\qquad c,d\in\mathbb R,\quad r\ge0,
		\]
		where, if $r=0$, the value of $\theta$ is immaterial.  Set
		\[
		D=\operatorname{diag}(1,e^{-i\theta}).
		\]
		Since $D$ is diagonal, it commutes with $V^*AV$, and hence
		\[
		D^*(V^*AV)D=V^*AV.
		\]
		On the other hand,
		\[
		D^*(V^*BV)D
		=
		\begin{pmatrix}
			c&r\\
			r&d
		\end{pmatrix}.
		\]
		Thus, with $U=VD$, the matrix $U^*AU$ is diagonal, while $U^*BU$ is
		real symmetric with nonnegative off-diagonal entry.
		
		Finally, unitary similarity preserves positive semidefiniteness.
		Therefore, if $A,B\ge0$, then also $U^*AU,U^*BU\ge0$.
	\end{proof}
\subsection{A singular-value identity for real $2\times2$ matrices}

The next identity isolates the two combinations of matrix entries that
determine the singular values of a real $2\times2$ matrix.  In the proof
of the main theorem, one of these quantities will remain unchanged while
the other can only decrease.

\begin{lemma}\label{lem:Rpm}
	Let
	\[
	M=
	\begin{pmatrix}
		\alpha&\beta\\
		\gamma&\delta
	\end{pmatrix}
	\in M_2(\mathbb R),
	\]
	and define
	\begin{align}
		R_+(M)
		&=
		\sqrt{(\alpha+\delta)^2+(\beta-\gamma)^2},
		\label{eq:Rplus}\\
		R_-(M)
		&=
		\sqrt{(\alpha-\delta)^2+(\beta+\gamma)^2}.
		\label{eq:Rminus}
	\end{align}
	Then
	\begin{equation}\label{eq:R-sing}
		s_1(M)=\frac{R_+(M)+R_-(M)}2,
		\qquad
		s_2(M)=\frac{|R_+(M)-R_-(M)|}{2}.
	\end{equation}
	In particular,
	\begin{equation}\label{eq:R-trace}
		\|M\|_1=s_1(M)+s_2(M)
		=\max\{R_+(M),R_-(M)\}.
	\end{equation}
	More precisely,
	\begin{equation}\label{eq:R-sign}
		\|M\|_1=
		\begin{cases}
			R_+(M),&\det M\ge0,\\
			R_-(M),&\det M<0.
		\end{cases}
	\end{equation}
\end{lemma}

\begin{proof}
	A direct expansion gives
	\begin{align*}
		R_+(M)^2
		&=
		(\alpha+\delta)^2+(\beta-\gamma)^2
		=\|M\|_F^2+2\det M,\\
		R_-(M)^2
		&=
		(\alpha-\delta)^2+(\beta+\gamma)^2
		=\|M\|_F^2-2\det M.
	\end{align*}
	In particular,
	\begin{equation}\label{eq:R-difference}
		R_+(M)^2-R_-(M)^2=4\det M.
	\end{equation}
	
	Let $s_1=s_1(M)\ge s_2=s_2(M)\ge0$.  Since
	\[
	s_1^2+s_2^2=\|M\|_F^2,
	\qquad
	s_1s_2=|\det M|,
	\]
	we obtain
	\begin{align*}
		(s_1+s_2)^2
		&=\|M\|_F^2+2|\det M|,\\
		(s_1-s_2)^2
		&=\|M\|_F^2-2|\det M|.
	\end{align*}
	
	If $\det M\ge0$, comparison with the formulas for $R_\pm(M)^2$
	shows, since all quantities are nonnegative, that
	\[
	R_+(M)=s_1+s_2,
	\qquad
	R_-(M)=s_1-s_2.
	\]
	If $\det M<0$, the roles are reversed:
	\[
	R_-(M)=s_1+s_2,
	\qquad
	R_+(M)=s_1-s_2.
	\]
	Thus in either case
	\[
	\{R_+(M),R_-(M)\}
	=
	\{s_1+s_2,\;s_1-s_2\}.
	\]
	Adding and subtracting these two quantities yields
	\[
	s_1=\frac{R_+(M)+R_-(M)}2,
	\qquad
	s_2=\frac{|R_+(M)-R_-(M)|}{2},
	\]
	which proves \eqref{eq:R-sing}.  Finally,
	\[
	\|M\|_1=s_1+s_2
	=\max\{R_+(M),R_-(M)\},
	\]
	and \eqref{eq:R-sign} follows directly from
	\eqref{eq:R-difference}.
\end{proof}
	
	\subsection{Weak majorization by the ordinary Heinz expression}
	
	We can now combine the phase reduction with the preceding
	singular-value identity.  The key point is that replacing the second
	Heinz summand by its transpose preserves one of the two quantities
	$R_\pm$, while positivity forces the other to decrease.
	
	\begin{theorem}[Two-dimensional weak majorization]\label{thm:2by2}
		Let $A,B\in M_2$ be positive definite and let $0<t<1$.  Then
		\begin{equation}\label{eq:major-2}
			s\!\left(C_t(A,B)\right)
			\prec_w
			s\!\left(H_t(A,B)\right).
		\end{equation}
		Equivalently, for every unitarily invariant norm,
		\begin{equation}\label{eq:strong-2}
			\UIN{A^tB^{1-t}+B^tA^{1-t}}
			\le
			\UIN{A^tB^{1-t}+A^{1-t}B^t}.
		\end{equation}
		Consequently,
		\begin{equation}\label{eq:Bourin2}
			\UIN{A^tB^{1-t}+B^tA^{1-t}}
			\le
			\UIN{A+B}
		\end{equation}
		for every unitarily invariant norm.
	\end{theorem}
	
	\begin{proof}
		By Lemma~\ref{lem:real-reduction}, simultaneous unitary conjugation,
		and, if necessary, an interchange of the two basis vectors, we may
		assume
		\[
		A=
		\begin{pmatrix}
			a&0\\
			0&b
		\end{pmatrix},
		\qquad a\ge b>0,
		\]
		and
		\[
		B=
		\begin{pmatrix}
			c&r\\
			r&d
		\end{pmatrix},
		\qquad r\ge0,
		\]
		with $B>0$.  Since functional calculus commutes with unitary
		conjugation, this reduction does not affect the singular values of
		either $C_t(A,B)$ or $H_t(A,B)$.
		
		Set
		\begin{equation}\label{eq:XY}
			X=A^tB^{1-t},
			\qquad
			Y=A^{1-t}B^t.
		\end{equation}
		Then
		\[
		H_t(A,B)=X+Y.
		\]
		The matrices $A^{1-t}$ and $B^t$ are real symmetric, and hence
		\[
		Y^T=(A^{1-t}B^t)^T=B^tA^{1-t}.
		\]
		Therefore
		\begin{equation}\label{eq:CXY}
			C_t(A,B)=X+Y^T.
		\end{equation}
		In particular, both $C_t(A,B)$ and $H_t(A,B)$ are real matrices, so
		Lemma~\ref{lem:Rpm} applies.
		
		Define
		\[
		\Delta_X=X_{12}-X_{21},
		\qquad
		\Delta_Y=Y_{12}-Y_{21}.
		\]
		We first show that
		\begin{equation}\label{eq:Deltas}
			\Delta_X\ge0,
			\qquad
			\Delta_Y\ge0.
		\end{equation}
		For $0<\alpha<1$, spectral functional calculus in dimension two gives
		\begin{equation}\label{eq:Bpower-linear}
			B^\alpha=u_\alpha B+v_\alpha I
		\end{equation}
		for suitable real scalars $u_\alpha,v_\alpha$.  Indeed, if the
		eigenvalues of $B$ are distinct, say
		$\lambda_1>\lambda_2>0$, interpolation at these two eigenvalues gives
		\[
		u_\alpha
		=
		\frac{\lambda_1^\alpha-\lambda_2^\alpha}
		{\lambda_1-\lambda_2}
		>0.
		\]
		Consequently,
		\[
		(B^\alpha)_{12}
		=
		u_\alpha B_{12}
		=
		u_\alpha r
		\ge0.
		\]
		If the two eigenvalues coincide, then $B$ is a scalar matrix, so its
		off-diagonal entry, and hence that of $B^\alpha$, is zero.  Thus in all
		cases
		\begin{equation}\label{eq:Balpha12}
			(B^\alpha)_{12}\ge0.
		\end{equation}
		
		Since $B^\alpha$ is symmetric, we obtain
		\begin{align}
			\Delta_X
			&=
			(a^t-b^t)(B^{1-t})_{12}\ge0,
			\label{eq:DX}\\
			\Delta_Y
			&=
			(a^{1-t}-b^{1-t})(B^t)_{12}\ge0.
			\label{eq:DY}
		\end{align}
		
		We now compare the two quantities from Lemma~\ref{lem:Rpm}.
		Since $C_t=X+Y^T$ and $H_t=X+Y$, transposing $Y$ leaves its
		diagonal entries unchanged.  Hence
		\[
		(C_t)_{11}=(H_t)_{11},
		\qquad
		(C_t)_{22}=(H_t)_{22}.
		\]
		Moreover, transposition also preserves the sum of the off-diagonal
		entries:
		\[
		(C_t)_{12}+(C_t)_{21}
		=
		(H_t)_{12}+(H_t)_{21}.
		\]
		It follows directly from the definition of $R_-$ that
		\begin{equation}\label{eq:Rminus-same}
			R_-(C_t)=R_-(H_t).
		\end{equation}
		
		For the skew off-diagonal component, however,
		\[
		(H_t)_{12}-(H_t)_{21}
		=
		\Delta_X+\Delta_Y,
		\]
		whereas
		\[
		(C_t)_{12}-(C_t)_{21}
		=
		\Delta_X-\Delta_Y.
		\]
		Since $\Delta_X,\Delta_Y\ge0$,
		\[
		|\Delta_X-\Delta_Y|
		\le
		\Delta_X+\Delta_Y.
		\]
		The diagonal entries of $C_t$ and $H_t$ are the same, and hence their
		traces are equal.  By the definition of $R_+$, we therefore obtain
		\begin{equation}\label{eq:Rplus-order}
			R_+(C_t)\le R_+(H_t).
		\end{equation}
		
		Combining \eqref{eq:Rminus-same} and \eqref{eq:Rplus-order} with
		Lemma~\ref{lem:Rpm} gives
		\[
		s_1(C_t)
		=
		\frac{R_+(C_t)+R_-(C_t)}2
		\le
		\frac{R_+(H_t)+R_-(H_t)}2
		=
		s_1(H_t).
		\]
		The trace-norm formula in Lemma~\ref{lem:Rpm} similarly yields
		\begin{align*}
			s_1(C_t)+s_2(C_t)
			&=
			\max\{R_+(C_t),R_-(C_t)\}\\
			&\le
			\max\{R_+(H_t),R_-(H_t)\}\\
			&=
			s_1(H_t)+s_2(H_t).
		\end{align*}
		These are precisely the two Ky Fan inequalities in dimension two.
		Hence
		\[
		s(C_t(A,B))\prec_w s(H_t(A,B)),
		\]
		which proves \eqref{eq:major-2}.  By the Ky Fan dominance theorem,
		this is equivalent to \eqref{eq:strong-2}.  Finally, the classical
		Heinz inequality \eqref{eq:Heinz} gives
		\[
		\UIN{C_t(A,B)}
		\le
		\UIN{H_t(A,B)}
		\le
		\UIN{A+B},
		\]
		and therefore \eqref{eq:Bourin2}.
	\end{proof}
	
	The positive-definite hypothesis in Theorem~\ref{thm:2by2} is only
	needed to carry out the argument without singular-power
	approximations.  Continuity extends the conclusion to the boundary of
	the positive cone.
	
	\begin{corollary}\label{cor:2PSD}
		The conclusions of Theorem~\ref{thm:2by2} remain valid for positive
		semidefinite $2\times2$ matrices.
	\end{corollary}
	
	\begin{proof}
		For $\varepsilon>0$, set
		\[
		A_\varepsilon=A+\varepsilon I,
		\qquad
		B_\varepsilon=B+\varepsilon I.
		\]
		Then $A_\varepsilon,B_\varepsilon>0$, so
		Theorem~\ref{thm:2by2} gives
		\[
		s\!\left(C_t(A_\varepsilon,B_\varepsilon)\right)
		\prec_w
		s\!\left(H_t(A_\varepsilon,B_\varepsilon)\right).
		\]
		As $\varepsilon\downarrow0$, continuity of the functional calculus
		for the functions $x^t$ and $x^{1-t}$ on $[0,\infty)$ gives
		\[
		C_t(A_\varepsilon,B_\varepsilon)\longrightarrow C_t(A,B),
		\qquad
		H_t(A_\varepsilon,B_\varepsilon)\longrightarrow H_t(A,B).
		\]
		Since singular values are continuous, the two Ky Fan inequalities
		pass to the limit.  Hence
		\[
		s(C_t(A,B))\prec_w s(H_t(A,B)),
		\]
		and the remaining conclusions follow by Ky Fan dominance and the
		Heinz inequality.
	\end{proof}
	
	\begin{remark}\label{rem:dimension-specific}
		Theorem~\ref{thm:2by2} is strictly stronger than the dimension-two
		instance of Bourin's inequality: it places the ordinary Heinz
		expression itself between the crossed-Heinz expression and $A+B$,
		in the sense that
		\[
		\UIN{C_t(A,B)}
		\le
		\UIN{H_t(A,B)}
		\le
		\UIN{A+B}
		\]
		for every unitarily invariant norm.  The first inequality is known to
		fail already for positive definite $3\times3$ matrices
		\cite{BELV}.  Thus the theorem reveals a genuinely two-dimensional
		rigidity phenomenon; the next result shows that dimension two is in
		fact the exact threshold for this stronger comparison.
	\end{remark}
	
The preceding theorem is not merely a low-dimensional special case:
together with the known $3\times3$ counterexample, it identifies
dimension two as the exact threshold for universal domination by the
ordinary Heinz expression.

\begin{corollary}[Sharp dimensional threshold]
	\label{cor:dimension-threshold}
	Let $n\ge1$.  The inequality
	\begin{equation}\label{eq:dimension-comparison}
		\UIN{C_t(A,B)}
		\le
		\UIN{H_t(A,B)}
	\end{equation}
	holds for every pair of positive semidefinite matrices
	$A,B\in M_n$, every $0<t<1$, and every unitarily invariant norm
	if and only if
	\[
	n\le2.
	\]
	Moreover, for every $n\ge3$ the failure can be realized by positive
	definite matrices and the operator norm.
\end{corollary}

\begin{proof}
	For $n=1$, the two expressions coincide, so the assertion is
	immediate.  For $n=2$, the result follows from
	Corollary~\ref{cor:2PSD}.
	
	For $n=3$, Bottazzi, Elencwajg, Larotonda, and Varela
	\cite[Example~3.1]{BELV} exhibit positive definite matrices
	$A_0,B_0\in M_3$ and a parameter $t_0\in(0,1)$ such that
	\begin{equation}\label{eq:BELV-counterexample}
		\|C_{t_0}(A_0,B_0)\|_\infty
		>
		\|H_{t_0}(A_0,B_0)\|_\infty.
	\end{equation}
	
	It remains to propagate this counterexample to every dimension
	$n>3$.  Choose $\varepsilon>0$ such that
	\[
	2\varepsilon
	<
	\|H_{t_0}(A_0,B_0)\|_\infty
	\]
	and set
	\[
	A=A_0\oplus\varepsilon I_{n-3},
	\qquad
	B=B_0\oplus\varepsilon I_{n-3}.
	\]
	Then $A,B>0$.  Since functional calculus respects direct sums,
	\[
	A^{t_0}
	=
	A_0^{t_0}\oplus\varepsilon^{t_0}I_{n-3},
	\qquad
	B^{1-t_0}
	=
	B_0^{1-t_0}\oplus\varepsilon^{1-t_0}I_{n-3},
	\]
	and similarly for the complementary powers.  Consequently,
	\begin{align*}
		C_{t_0}(A,B)
		&=
		C_{t_0}(A_0,B_0)\oplus2\varepsilon I_{n-3},\\
		H_{t_0}(A,B)
		&=
		H_{t_0}(A_0,B_0)\oplus2\varepsilon I_{n-3}.
	\end{align*}
	For the operator norm,
	\[
	\|X\oplus2\varepsilon I_{n-3}\|_\infty
	=
	\max\{\|X\|_\infty,2\varepsilon\}.
	\]
	By \eqref{eq:BELV-counterexample} and the choice of $\varepsilon$,
	\[
	2\varepsilon
	<
	\|H_{t_0}(A_0,B_0)\|_\infty
	<
	\|C_{t_0}(A_0,B_0)\|_\infty.
	\]
	Hence
	\begin{align*}
		\|C_{t_0}(A,B)\|_\infty
		&=
		\|C_{t_0}(A_0,B_0)\|_\infty,\\
		\|H_{t_0}(A,B)\|_\infty
		&=
		\|H_{t_0}(A_0,B_0)\|_\infty,
	\end{align*}
	and therefore
	\[
	\|C_{t_0}(A,B)\|_\infty
	>
	\|H_{t_0}(A,B)\|_\infty.
	\]
	Thus universal validity fails in every dimension $n\ge3$, completing
	the proof.
\end{proof}

\begin{remark}\label{rem:threshold-vs-Bourin}
	Corollary~\ref{cor:dimension-threshold} concerns the stronger
	comparison with the ordinary Heinz expression, not Bourin's original
	inequality \eqref{eq:Bourin}.  Thus the failure in dimensions
	$n\ge3$ rules out the ordinary Heinz expression as a universal
	intermediate bound, but it does not provide a counterexample to
	Bourin's inequality itself.  This distinction motivates the remainder
	of the paper: having determined the exact dimensional threshold for
	the stronger comparison, we now turn to higher-dimensional mechanisms
	that address Bourin's inequality directly.
\end{remark}
	
\section{Exact second variation of the Ky Fan deficits}

Having identified the obstruction to the Heinz-mediated approach in
dimensions $n\ge3$, we now return directly to Bourin's inequality in
arbitrary dimension.  The first question is whether the inequality is
locally stable near the equality configuration $A=B$.  The following
calculation shows that, for every Ky Fan norm, the deficit has a
nonnegative quadratic leading term whose coefficients can be computed
explicitly.

Fix $A>0$ with simple spectrum
\[
\lambda_1>\lambda_2>\cdots>\lambda_n>0,
\]
and work in an orthonormal eigenbasis of $A$.  For $1\le k\le n$, let
$P_k$ denote the spectral projection onto the span of the first $k$
eigenvectors.  For a sufficiently small Hermitian perturbation $E$,
define the $k$th Ky Fan deficit by
\begin{equation}\label{eq:F}
	F_{k,t}(E)
	=
	\KF{2A+E}{k}
	-
	\KF{C_t(A,A+E)}{k}.
\end{equation}
Here $2A+E=A+(A+E)$ is precisely the right-hand side appearing in
Bourin's inequality.

Define
\begin{equation}\label{eq:phi}
	\phi_t(x,y)
	=
	\begin{cases}
		\displaystyle
		\frac{(x^t-y^t)(x^{1-t}-y^{1-t})}{(x-y)^2},
		&x\ne y,\\[9pt]
		\displaystyle
		\frac{t(1-t)}{x},
		&x=y.
	\end{cases}
\end{equation}
Since both functions $x\mapsto x^t$ and $x\mapsto x^{1-t}$ are
strictly increasing on $(0,\infty)$, one has
\begin{equation}\label{eq:phi-positive}
	\phi_t(x,y)>0
	\qquad (x,y>0,\;0<t<1).
\end{equation}

\begin{theorem}[Exact second variation]\label{thm:variation}
	For every Hermitian $H$, every $1\le k\le n$, and every $0<t<1$,
	\begin{equation}\label{eq:variation}
		F_{k,t}(\varepsilon H)
		=
		\varepsilon^2
		\sum_{i=1}^k\sum_{j=1}^n
		\phi_t(\lambda_i,\lambda_j)|H_{ij}|^2
		+O(\varepsilon^3)
		\qquad (\varepsilon\to0).
	\end{equation}
	For fixed $A$, the remainder is uniform when $H$ ranges over a bounded
	set of Hermitian matrices and $t$ ranges over a compact subinterval of
	$(0,1)$.
\end{theorem}

\begin{proof}
	For $0<t<1$, let
	\[
	f_t(x)=x^t,
	\qquad
	L_t=Df_t(A).
	\]
	The divided-difference formula for the Fr\'echet derivative of the
	functional calculus gives
	\begin{equation}\label{eq:Frechet}
		(L_t(H))_{ij}
		=
		f_t^{[1]}(\lambda_i,\lambda_j)H_{ij},
	\end{equation}
	where
	\[
	f_t^{[1]}(x,y)
	=
	\begin{cases}
		\displaystyle\frac{x^t-y^t}{x-y},&x\ne y,\\[7pt]
		tx^{t-1},&x=y.
	\end{cases}
	\]
	
	We use the exact defect identity \eqref{eq:defect},
	\[
	A+B-C_t(A,B)
	=
	(A^t-B^t)(A^{1-t}-B^{1-t}),
	\]
	with $B=A+\varepsilon H$.  Fr\'echet differentiability gives
	\[
	A^t-(A+\varepsilon H)^t
	=
	-\varepsilon L_t(H)+O(\varepsilon^2)
	\]
	and
	\[
	A^{1-t}-(A+\varepsilon H)^{1-t}
	=
	-\varepsilon L_{1-t}(H)+O(\varepsilon^2).
	\]
	Multiplying these expansions yields
	\[
	(A^t-(A+\varepsilon H)^t)
	(A^{1-t}-(A+\varepsilon H)^{1-t})
	=
	\varepsilon^2L_t(H)L_{1-t}(H)
	+O(\varepsilon^3).
	\]
	Since $A+B=2A+\varepsilon H$, the defect identity therefore gives
	\begin{equation}\label{eq:Cexp}
		C_t(A,A+\varepsilon H)
		=
		2A+\varepsilon H
		-
		\varepsilon^2L_t(H)L_{1-t}(H)
		+O(\varepsilon^3).
	\end{equation}
	
	Now set
	\[
	g_k(X)=\KF{X}{k}.
	\]
	Because $2A$ is positive definite and has distinct singular values
	$2\lambda_1>\cdots>2\lambda_n>0$, the sum of its largest $k$
	singular values is smooth in a neighborhood of $2A$.  For $k<n$,
	this follows from the strict gap
	\[
	2\lambda_k>2\lambda_{k+1},
	\]
	while for $k=n$ the trace norm is smooth in a neighborhood of the
	invertible matrix $2A$.  Its first derivative at $2A$ is
	\begin{equation}\label{eq:dg}
		Dg_k(2A)[Z]
		=
		\Ree\Tr(P_kZ).
	\end{equation}
	
	For brevity, write
	\[
	Q_t(H)=L_t(H)L_{1-t}(H).
	\]
	By \eqref{eq:Cexp}, the two arguments of $g_k$ occurring in
	\eqref{eq:F} differ by
	\[
	\varepsilon^2Q_t(H)+O(\varepsilon^3).
	\]
	Since $Dg_k$ is continuously differentiable near $2A$,
	\[
	Dg_k(2A+\varepsilon H)
	=
	Dg_k(2A)+O(\varepsilon).
	\]
	Consequently, Taylor expansion gives
	\begin{align*}
		F_{k,t}(\varepsilon H)
		&=
		g_k(2A+\varepsilon H)
		-
		g_k\!\left(
		2A+\varepsilon H-\varepsilon^2Q_t(H)+O(\varepsilon^3)
		\right)\\
		&=
		\varepsilon^2Dg_k(2A)[Q_t(H)]
		+O(\varepsilon^3)\\
		&=
		\varepsilon^2
		\Ree\Tr\!\left(P_kL_t(H)L_{1-t}(H)\right)
		+O(\varepsilon^3).
	\end{align*}
	
	It remains to evaluate the coefficient.  Since the divided differences
	in \eqref{eq:Frechet} are real and symmetric in their two arguments,
	and $H_{ji}=\overline{H_{ij}}$, we obtain
	\begin{align*}
		\Ree\Tr\!\left(P_kL_t(H)L_{1-t}(H)\right)
		&=
		\sum_{i=1}^k\sum_{j=1}^n
		f_t^{[1]}(\lambda_i,\lambda_j)
		f_{1-t}^{[1]}(\lambda_i,\lambda_j)
		|H_{ij}|^2\\
		&=
		\sum_{i=1}^k\sum_{j=1}^n
		\phi_t(\lambda_i,\lambda_j)|H_{ij}|^2.
	\end{align*}
	Indeed, for $i\ne j$ the product of the two divided differences is
	exactly
	\[
	\frac{(\lambda_i^t-\lambda_j^t)
		(\lambda_i^{1-t}-\lambda_j^{1-t})}
	{(\lambda_i-\lambda_j)^2},
	\]
	whereas for $i=j$ it equals
	\[
	t\lambda_i^{t-1}(1-t)\lambda_i^{-t}
	=
	\frac{t(1-t)}{\lambda_i}.
	\]
	This proves \eqref{eq:variation}.
	
	For the uniformity assertion, let $J\Subset(0,1)$ be compact and
	choose a neighborhood $\mathcal U$ of $A$ whose closure consists of
	positive definite matrices.  The map
	\[
	(X,t)\longmapsto X^t
	\]
	is jointly smooth on $\overline{\mathcal U}\times J$, with the
	derivatives required above uniformly bounded there.  Moreover, after
	shrinking $\mathcal U$ if necessary, the singular-value gap defining
	the first $k$ singular values remains uniformly separated for
	$k<n$, while all matrices involved remain uniformly invertible for
	$k=n$.  Hence the relevant derivatives of $g_k$ are uniformly bounded
	as well.  It follows that the remainder in \eqref{eq:variation} is
	$O(\varepsilon^3)$ uniformly for $t\in J$ and for $H$ in bounded
	subsets of the Hermitian matrices.
\end{proof}

\begin{remark}\label{rem:variation-kernel}
	The quadratic form in Theorem~\ref{thm:variation} is positive
	semidefinite.  Since every coefficient $\phi_t(\lambda_i,\lambda_j)$
	is strictly positive, its null space is described explicitly by
	\[
	H_{ij}=0
	\qquad
	(1\le i\le k,\;1\le j\le n).
	\]
	For Hermitian $H$, this is equivalent to
	\[
	P_kH=HP_k=0,
	\]
	or, writing $Q_k=I-P_k$,
	\[
	H=Q_kHQ_k.
	\]
	Thus for $k<n$ the second variation has a nontrivial null space,
	whereas for $k=n$ it is positive definite.  Positivity of the second
	variation alone is therefore insufficient to establish the desired
	local inequality when $k<n$.  The next section shows that these null
	directions form an exact zero-deficit manifold and then controls the
	transverse directions uniformly.
\end{remark}
	
	\section{Local validity in arbitrary dimension}
	
	The null directions of the second variation are not an obstruction to
	local validity: they form an exact manifold on which the Ky Fan deficit
	vanishes identically, while the deficit grows quadratically in every
	transverse direction.  This yields a local form of Bourin's inequality
	in arbitrary dimension.
	
	Throughout this section,
	\[
	J\Subset(0,1)
	\]
	means that $J$ is a compact subinterval of $(0,1)$; equivalently, $J$
	stays a positive distance away from the endpoints $0$ and $1$.
	
	\begin{theorem}[Local validity]\label{thm:local}
		Let $A>0$ have simple spectrum and let $J\Subset(0,1)$.  There exist
		constants $r,c>0$, depending only on $A$ and $J$, such that for every
		Hermitian $E$ satisfying $\|E\|_F<r$, every $t\in J$, and every
		$1\le k\le n$,
		\begin{equation}\label{eq:local}
			A+E>0,
			\qquad
			F_{k,t}(E)\ge c\|P_kE\|_F^2.
		\end{equation}
		Consequently,
		\begin{equation}\label{eq:localUIN}
			\UIN{C_t(A,A+E)}
			\le
			\UIN{2A+E}
		\end{equation}
		for every unitarily invariant norm.
		
		Moreover, within this neighborhood,
		\[
		F_{k,t}(E)=0
		\quad\Longleftrightarrow\quad
		P_kE=0.
		\]
		For $k=n$, equality occurs only when $E=0$.
	\end{theorem}
	
	\begin{proof}
		Fix $k$ and write
		\[
		P=P_k,
		\qquad
		Q=I-P.
		\]
		Decompose
		\begin{equation}\label{eq:Edec}
			E=E_0+W,
			\qquad
			E_0=QEQ,
			\qquad
			W=E-QEQ.
		\end{equation}
		Then
		\[
		QWQ=0,
		\qquad
		PE_0=E_0P=0,
		\]
		and $E_0$ and $W$ are orthogonal in the Frobenius inner product.
		Moreover,
		\begin{equation}\label{eq:PWPE}
			PW=PE.
		\end{equation}
		
		We first identify the exact zero-deficit directions.  Since $A$
		commutes with $P$ and $Q$, both $A$ and $A+E_0$ are block diagonal
		relative to
		\[
		P\mathbb C^n\oplus Q\mathbb C^n,
		\]
		and their $P$-blocks coincide:
		\[
		P(A+E_0)P=PAP.
		\]
		Functional calculus preserves this block decomposition and therefore
		\[
		P\,C_t(A,A+E_0)\,P=2PAP.
		\]
		Likewise,
		\[
		P(2A+E_0)P=2PAP.
		\]
		Thus the two matrices $C_t(A,A+E_0)$ and $2A+E_0$ have the same
		$P$-block, namely $2PAP$.
		
		Suppose first that $k<n$.  At $E_0=0$, the singular values associated
		with this block are
		\[
		2\lambda_1,\ldots,2\lambda_k,
		\]
		and they are separated from the remaining singular values by the
		strict gap
		\[
		2\lambda_k-2\lambda_{k+1}>0.
		\]
		By continuity, after restricting $\|E_0\|_F$ to a sufficiently small
		neighborhood, the singular values arising from the $Q$-blocks of both
		$2A+E_0$ and $C_t(A,A+E_0)$ remain strictly below those of the
		common $P$-block.  This can be done uniformly for $t\in J$, since $J$
		is compact.  Hence the largest $k$ singular values of both matrices
		are precisely those of $2PAP$, and consequently
		\begin{equation}\label{eq:Fzero}
			F_{k,t}(E_0)=0.
		\end{equation}
		For $k=n$, one has $Q=0$ and hence $E_0=0$, so
		\eqref{eq:Fzero} is immediate.
		
		We next show that the first variation also vanishes along this
		zero-deficit manifold:
		\begin{equation}\label{eq:DFzero}
			DF_{k,t}(E_0)=0.
		\end{equation}
		To see this, put
		\[
		B=A+E_0.
		\]
		Since $A$ and $B$ have the same $P$-block, functional calculus gives
		\[
		P(A^t-B^t)=
		(A^t-B^t)P=0
		\]
		and similarly
		\[
		P(A^{1-t}-B^{1-t})
		=
		(A^{1-t}-B^{1-t})P=0.
		\]
		Thus both factors in the defect identity
		\[
		A+B-C_t(A,B)
		=
		(A^t-B^t)(A^{1-t}-B^{1-t})
		\]
		are supported entirely in the $Q$-block.
		
		Differentiate this product with respect to $B$ in an arbitrary
		Hermitian direction $H$.  The derivative is the sum of two terms,
		each containing one factor supported in the $Q$-block.  Consequently,
		its $P$-diagonal block vanishes:
		\[
		P\,D_B\!\left[
		(A^t-B^t)(A^{1-t}-B^{1-t})
		\right][H]\,P=0.
		\]
		At the block-diagonal base point $E_0$, the largest $k$ singular
		values of both matrices are carried by the common positive $P$-block.
		Hence the differential of the corresponding Ky Fan functional is
		the real trace of the $P$-diagonal block of the variation.  The
		preceding identity therefore implies that the first derivatives of
		the two terms defining $F_{k,t}$ agree, proving
		\eqref{eq:DFzero}.
		
		We now estimate the deficit in directions transverse to the
		zero-deficit manifold.  Define
		\begin{equation}\label{eq:mmin}
			m=
			\min_{\substack{1\le i,j\le n\\ t\in J}}
			\phi_t(\lambda_i,\lambda_j).
		\end{equation}
		By \eqref{eq:phi-positive}, continuity, and compactness of $J$,
		\[
		m>0.
		\]
		For a Hermitian matrix $W$ satisfying $QWQ=0$,
		Theorem~\ref{thm:variation} gives
		\[
		D^2F_{k,t}(0)[W,W]
		=
		2\sum_{i=1}^k\sum_{j=1}^n
		\phi_t(\lambda_i,\lambda_j)|W_{ij}|^2
		\ge
		2m\|PW\|_F^2.
		\]
		Writing $W$ relative to $P\oplus Q$ as
		\[
		W=
		\begin{pmatrix}
			PWP&PWQ\\
			QWP&0
		\end{pmatrix},
		\]
		and using $QWP=(PWQ)^*$, we have
		\[
		\|W\|_F^2
		=
		\|PWP\|_F^2+2\|PWQ\|_F^2,
		\]
		whereas
		\[
		\|PW\|_F^2
		=
		\|PWP\|_F^2+\|PWQ\|_F^2.
		\]
		Therefore
		\[
		2\|PW\|_F^2\ge\|W\|_F^2,
		\]
		and hence
		\begin{equation}\label{eq:Hesslower}
			D^2F_{k,t}(0)[W,W]
			\ge
			m\|W\|_F^2.
		\end{equation}
		
		Let
		\[
		\mathcal T_k
		=
		\{W=W^*:QWQ=0\}.
		\]
		The Frobenius unit sphere of $\mathcal T_k$ is compact.  By taking a
		sufficiently small neighborhood of the origin, the singular-value
		clusters defining the relevant Ky Fan functionals remain separated,
		uniformly for $t\in J$.  In this neighborhood the map
		\[
		(Z,t,W)
		\longmapsto
		D^2F_{k,t}(Z)[W,W]
		\]
		is continuous.  Combining this continuity with
		\eqref{eq:Hesslower}, compactness of $J$, and compactness of the unit
		sphere of $\mathcal T_k$, we may shrink the neighborhood so that
		\begin{equation}\label{eq:Hessuniform}
			D^2F_{k,t}(Z)[W,W]
			\ge
			\frac{m}{2}\|W\|_F^2
		\end{equation}
		for every $t\in J$, every $W\in\mathcal T_k$, and every sufficiently
		small Hermitian $Z$.
		
		Since $E=E_0+W$, Taylor's formula along the segment
		\[
		E_0+sW,\qquad 0\le s\le1,
		\]
		together with \eqref{eq:Fzero} and \eqref{eq:DFzero}, yields
		\begin{align*}
			F_{k,t}(E)
			&=
			\int_0^1(1-s)
			D^2F_{k,t}(E_0+sW)[W,W]\,ds\\
			&\ge
			\frac{m}{4}\|W\|_F^2\\
			&\ge
			\frac{m}{4}\|PE\|_F^2,
		\end{align*}
		where the last inequality follows from \eqref{eq:PWPE}.  Since there
		are only finitely many values of $k$, we may choose a common radius
		for $k=1,\ldots,n$.  Shrinking it once more if necessary ensures
		$A+E>0$.  Thus \eqref{eq:local} holds, for example with
		\[
		c=\frac{m}{4}.
		\]
		
		Since \eqref{eq:local} holds for every $k$, we obtain
		\[
		\KF{C_t(A,A+E)}{k}
		\le
		\KF{2A+E}{k},
		\qquad 1\le k\le n.
		\]
		Equivalently,
		\[
		s(C_t(A,A+E))
		\prec_w
		s(2A+E).
		\]
		The Ky Fan dominance theorem therefore gives
		\eqref{eq:localUIN} for every unitarily invariant norm.
		
		Finally, if $F_{k,t}(E)=0$, then \eqref{eq:local} implies
		$PE=0$.  Since $E$ is Hermitian, also $EP=0$, and hence
		\[
		E=QEQ.
		\]
		Conversely, if $PE=0$, then Hermiticity gives $EP=0$, so $E=QEQ$;
		thus \eqref{eq:Fzero} gives $F_{k,t}(E)=0$.  This proves
		\[
		F_{k,t}(E)=0
		\quad\Longleftrightarrow\quad
		P_kE=0.
		\]
		For $k=n$, one has $P_n=I$, and therefore equality forces $E=0$.
	\end{proof}
	
	\begin{remark}\label{rem:local-scope}
		Theorem~\ref{thm:local} is genuinely local.  The compact restriction
		$J\Subset(0,1)$ is natural at the level of the present argument,
		since the coefficient
		\[
		\phi_t(x,x)=\frac{t(1-t)}{x}
		\]
		degenerates as $t\to0$ or $t\to1$.  Thus the coercivity constant
		obtained above cannot remain bounded away from zero by this method as
		$t$ approaches either endpoint.  The simple-spectrum assumption is
		used to ensure separation and smoothness of the relevant Ky Fan
		singular-value clusters; removing that hypothesis would require a
		separate treatment of spectral multiplicities.
	\end{remark}
	
\section{A sharp global refinement for a common two-point spectrum}

The local result above relies on perturbative control near $A=B$.
We now exhibit a genuinely global higher-dimensional class for which
Bourin's inequality admits a sharp positive remainder in the Loewner
order.  The class consists of positive matrices having the same
two-point spectrum, with no commutativity assumption on their spectral
projections.

\begin{theorem}[Sharp common two-point bound]\label{thm:binary}
	Let $a>b>0$, let $P,Q$ be orthogonal projections, and define
	\begin{equation}\label{eq:ABPQ}
		A=bI+(a-b)P,
		\qquad
		B=bI+(a-b)Q.
	\end{equation}
	For $0<t<1$, put
	\begin{equation}\label{eq:dt}
		d_t=(a^t-b^t)(a^{1-t}-b^{1-t}),
	\end{equation}
	\begin{equation}\label{eq:L}
		L=(\sqrt a+\sqrt b)^2,
	\end{equation}
	and
	\begin{equation}\label{eq:gamma}
		\gamma_t(a,b)
		=
		\frac{\min\{d_t,L-d_t\}}{(a-b)^2}.
	\end{equation}
	Then $\gamma_t(a,b)>0$ and
	\begin{equation}\label{eq:Loewner}
		|C_t(A,B)|
		+
		\gamma_t(a,b)(A-B)^2
		\preceq
		A+B.
	\end{equation}
	Moreover, $\gamma_t(a,b)$ is the largest coefficient for which
	\eqref{eq:Loewner} holds simultaneously for all projection pairs
	$P,Q$ in all dimensions.  The optimality is already witnessed by
	rank-one projections in dimension two.  In particular,
	\begin{equation}\label{eq:binaryUIN}
		\UIN{C_t(A,B)}
		\le
		\UIN{A+B}
	\end{equation}
	for every unitarily invariant norm.
\end{theorem}

\begin{proof}
	Set
	\[
	\delta=a-b>0,
	\qquad
	X=P+Q.
	\]
	Since $P$ and $Q$ are projections,
	\begin{align}
		(P-Q)^2
		&=P+Q-PQ-QP \notag\\
		&=X(2I-X).
		\label{eq:PQ}
	\end{align}
	In particular, $(P-Q)^2$ is a polynomial in the self-adjoint matrix
	$X$.
	
	For every $0<\alpha<1$,
	\[
	A^\alpha=b^\alpha I+(a^\alpha-b^\alpha)P,
	\qquad
	B^\alpha=b^\alpha I+(a^\alpha-b^\alpha)Q.
	\]
	Hence
	\[
	A^t-B^t=(a^t-b^t)(P-Q)
	\]
	and
	\[
	A^{1-t}-B^{1-t}
	=
	(a^{1-t}-b^{1-t})(P-Q).
	\]
	The defect identity \eqref{eq:defect} therefore gives
	\[
	A+B-C_t(A,B)
	=
	d_t(P-Q)^2.
	\]
	Since
	\[
	A+B=2bI+\delta X,
	\]
	we obtain
	\begin{equation}\label{eq:Cbinary}
		C_t(A,B)
		=
		2bI+\delta X-d_tX(2I-X).
	\end{equation}
	Moreover,
	\begin{equation}\label{eq:Sbinary}
		A+B=2bI+\delta X
	\end{equation}
	and
	\begin{equation}\label{eq:Dbinary}
		(A-B)^2
		=
		\delta^2(P-Q)^2
		=
		\delta^2X(2I-X).
	\end{equation}
	
	Thus $C_t(A,B)$ is self-adjoint in this class, and
	$C_t(A,B)$, $A+B$, and $(A-B)^2$ are all real-valued functions of
	the same self-adjoint matrix $X$.  Since $0\preceq P+Q\preceq2I$,
	\[
	\operatorname{spec}(X)\subset[0,2].
	\]
	Consequently, \eqref{eq:Loewner} reduces by spectral functional
	calculus to a scalar inequality on $[0,2]$.
	
	For $0\le x\le2$, set
	\[
	s(x)=2b+\delta x,
	\qquad
	u(x)=x(2-x).
	\]
	At a spectral value $x$, the desired inequality with coefficient
	$\gamma$ becomes
	\[
	|s(x)-d_tu(x)|+\gamma\delta^2u(x)
	\le s(x).
	\]
	For $0<x<2$, where $u(x)>0$, the largest admissible coefficient at
	$x$ is therefore
	\begin{align}
		\frac{s(x)-|s(x)-d_tu(x)|}{\delta^2u(x)}
		&=
		\frac1{\delta^2}
		\min\left\{
		d_t,\,
		\frac{2s(x)}{u(x)}-d_t
		\right\}.
		\label{eq:quotient}
	\end{align}
	Here we used the elementary identity
	\[
	s-|s-y|=\min\{y,\,2s-y\}.
	\]
	
	We next minimize the only $x$-dependent term.  Since
	\[
	\frac{2s(x)}{u(x)}
	=
	\frac{2(2b+\delta x)}{x(2-x)},
	\]
	a direct differentiation shows that
	\begin{equation}\label{eq:minratio}
		\min_{0<x<2}\frac{2s(x)}{u(x)}
		=
		(\sqrt a+\sqrt b)^2
		=
		L,
	\end{equation}
	with the unique minimizer
	\begin{equation}\label{eq:xstar}
		x_*
		=
		\frac{2\sqrt b}{\sqrt a+\sqrt b}.
	\end{equation}
	Also,
	\begin{align}
		L-d_t
		&=
		a+b+2\sqrt{ab}
		-
		(a^t-b^t)(a^{1-t}-b^{1-t}) \notag\\
		&=
		2\sqrt{ab}
		+a^tb^{1-t}
		+a^{1-t}b^t
		>0.
		\label{eq:Lminus}
	\end{align}
	Since $d_t>0$, it follows that
	\[
	\gamma_t(a,b)
	=
	\frac{\min\{d_t,L-d_t\}}{\delta^2}
	>0.
	\]
	Moreover, \eqref{eq:minratio} implies that the expression in
	\eqref{eq:quotient} is at least $\gamma_t(a,b)$ for every
	$0<x<2$.  At $x=0$ and $x=2$ one has $u(x)=0$, and the undivided
	scalar inequality is an equality.  Spectral functional calculus
	therefore proves \eqref{eq:Loewner}.
	
	It remains to prove optimality.  There are two cases.
	
	Suppose first that
	\[
	d_t\le\frac{L}{2}.
	\]
	Then, for every $0<x<2$,
	\[
	\frac{2s(x)}{u(x)}-d_t
	\ge
	L-d_t
	\ge
	d_t.
	\]
	Thus the quotient in \eqref{eq:quotient} is identically
	$d_t/\delta^2$.  In dimension two, take
	\[
	P=
	\begin{pmatrix}
		1&0\\
		0&0
	\end{pmatrix},
	\qquad
	Q=
	\begin{pmatrix}
		0&0\\
		0&1
	\end{pmatrix}.
	\]
	Then $P+Q=I$, so $x=1$ occurs in the spectrum of $X$, and
	\eqref{eq:quotient} equals
	\[
	\frac{d_t}{\delta^2}
	=
	\gamma_t(a,b).
	\]
	Hence no larger universal coefficient is possible.
	
	Suppose next that
	\[
	d_t>\frac{L}{2}.
	\]
	Since $a>b>0$, the number
	\[
	x_*=\frac{2\sqrt b}{\sqrt a+\sqrt b}
	\]
	lies in $(0,1)$.  Let $P$ be the rank-one projection onto
	$e_1=(1,0)^T$, and let $Q$ be the rank-one projection onto
	\[
	v=(r,\sqrt{1-r^2})^T,
	\qquad
	r=1-x_*\in(0,1).
	\]
	For two rank-one projections in dimension two, the eigenvalues of
	$P+Q$ are
	\[
	1\pm|\langle e_1,v\rangle|
	=
	1\pm r,
	\]
	and hence in the present case they are
	\[
	x_*
	\qquad\text{and}\qquad
	2-x_*.
	\]
	At the spectral value $x_*$, \eqref{eq:minratio} gives
	\[
	\frac{2s(x_*)}{u(x_*)}=L.
	\]
	Since $d_t>L/2$, the quotient in \eqref{eq:quotient} at $x_*$ is
	\[
	\frac{L-d_t}{\delta^2}
	=
	\gamma_t(a,b).
	\]
	Thus no coefficient larger than $\gamma_t(a,b)$ can hold for all
	projection pairs.  This proves the asserted sharpness, already in
	dimension two.
	
	Finally, dropping the nonnegative remainder in \eqref{eq:Loewner}
	gives
	\[
	|C_t(A,B)|\preceq A+B.
	\]
	Both sides are positive, so eigenvalue monotonicity yields
	\[
	\lambda_j(|C_t(A,B)|)
	\le
	\lambda_j(A+B),
	\qquad 1\le j\le n.
	\]
	Since
	\[
	\lambda_j(|C_t(A,B)|)=s_j(C_t(A,B)),
	\]
	we obtain
	\[
	s(C_t(A,B))\prec_w s(A+B).
	\]
	The Ky Fan dominance theorem therefore gives
	\eqref{eq:binaryUIN} for every unitarily invariant norm.
\end{proof}
	
	Taking traces in the sharp Loewner refinement yields a quantitative
	stability estimate: within the common two-point-spectrum class, the
	trace-norm deficit controls the squared Frobenius distance between the
	two matrices.
	
	\begin{corollary}\label{cor:trace-refinement}
		Under the assumptions of Theorem~\ref{thm:binary},
		\begin{equation}\label{eq:trace-refinement}
			\Tr(A+B)-\|C_t(A,B)\|_1
			\ge
			\gamma_t(a,b)\|A-B\|_F^2.
		\end{equation}
	\end{corollary}
	
	\begin{proof}
		Taking traces in \eqref{eq:Loewner} gives
		\[
		\Tr|C_t(A,B)|
		+
		\gamma_t(a,b)\Tr\!\left((A-B)^2\right)
		\le
		\Tr(A+B).
		\]
		Since
		\[
		\Tr|C_t(A,B)|=\|C_t(A,B)\|_1
		\]
		and, because $A-B$ is Hermitian,
		\[
		\Tr\!\left((A-B)^2\right)
		=
		\|A-B\|_F^2,
		\]
		the desired inequality follows.
	\end{proof}
	
	\begin{remark}\label{rem:trace-optimality}
		The optimality assertion in Theorem~\ref{thm:binary} concerns the
		coefficient in the Loewner-order refinement \eqref{eq:Loewner}.
		Although \eqref{eq:trace-refinement} is obtained from that sharp
		matrix inequality by taking traces, this does not by itself imply
		that $\gamma_t(a,b)$ is the optimal coefficient in the corresponding
		trace-deficit inequality.  No separate trace-level optimality claim
		is made here.
	\end{remark}

	\section*{Declarations}
	
	\noindent
	\textbf{Funding.}
	No external funding was received for this work.
	
	\medskip
	
	\noindent
	\textbf{Competing interests.}
	The author declares that there are no competing interests.
	
	\medskip
	
	\noindent
	\textbf{Data availability.}
	No external datasets were used in this study.  The results are
	theoretical and all arguments needed for the stated conclusions are
	contained in the manuscript.
	
	\medskip
	
	\noindent
	\textbf{Author contributions.}
	The author is solely responsible for the mathematical development and
	preparation of the manuscript.
	

\end{document}